\documentclass[11pt]{article}
\usepackage{amsmath,amsthm}
\usepackage{amssymb}
\usepackage[dvipdfmx]{graphicx,color}
\allowdisplaybreaks
\nopagebreak[3]
\newcommand{\newsection}[1]
{\section{#1}\setcounter{theorem}{0} \setcounter{equation}{0} \par\noindent}

\newtheorem{theorem}{Theorem}

\newtheorem{lemma}[theorem]{Lemma}
\newtheorem{corollary}[theorem]{Corollary}

\newtheorem{remark}[theorem]{Remark}
\newtheorem{claim}[theorem]{Claim}
\newcommand{\beq}{ \begin{equation} }
\newcommand{\eeq}{ \end{equation} }

\newcommand{\br}{{\mathbb R}}
\newcommand{\bc}{{\mathbb C}}

\newcommand{\vep}{ \varepsilon }
\newcommand{\brn}{ { \mathbb{R}^n } }
\renewcommand{\Re}{\operatorname{Re}}

\newcommand{\rmtwo}{{\rm{I}\hspace{-1pt}\rm{I}}}
\newcommand{\rmthree}{{\rm{I}\hspace{-1pt}\rm{I}\hspace{-1pt}\rm{I}}}
\newcommand{\rmfour}{{\rm{I}\hspace{-1pt}\rm{V}}}
\newcommand{\Sin}{\mathrm{Sin}^{-1}}

\title{
Finite time blow-up
of semi-linear Klein-Gordon equations \\
with positive initial energy
in FLRW spacetimes
}
\author
{
Makoto NAKAMURA
\thanks
{Graduate School of Information Science and Technology,  
Osaka University, 
1-5 Yamadaoka, Suita, Osaka 565-0871, JAPAN.  
E-mail:  \texttt{makoto.nakamura.ist@osaka-u.ac.jp} }
\ \ and \ 
Takuma YOSHIZUMI
\thanks
{Graduate School of Information Science and Technology,  
Osaka University, 
1-5 Yamadaoka, Suita, Osaka 565-0871, JAPAN.  
E-mail:  \texttt{yoshizumi.takuma@ist.osaka-u.ac.jp} }
}

\date{}

\begin{document}

\maketitle

\begin{abstract}
Blowing-up solutions for semi-linear Klein-Gordon equations are considered in Friedmann-Lema\^itre-Robertson-Walker spacetimes. Some sufficient conditions are shown by
applying the concavity method for semi-linear wave equations in the Minkowski spacetime to semi-linear Klein-Gordon equations in FLRW spacetimes.
\end{abstract}

\noindent
{\it Mathematics Subject Classification (2020)}: 
Primary 35L05; Secondary 35L71, 35Q75. 
%35G25 Initial value problems for nonlinear higher-order PDEs
%35L05 Wave equation
%35L15 Initial value problems for second-order hyperbolic equations
%35L70 Second-order nonlinear hyperbolic equations
%35L71 Second-order semilinear hyperbolic equations
%35L72 Second-order quasilinear hyperbolic equations
%35Q60 PDEs in connection with optics and electromagnetic theory
%35Q61 Maxwell equations
%35Q30 Navier-Stokes equations
%35Q75 PDEs in connection with relativity and gravitational theory
%74B05 Classical linear elasticity
%35G20 Nonlinear higher-order equations
%35Q76 Einstein equations

\vspace{5pt}

\noindent
{\it Keywords} : 
semilinear Klein-Gordon equation, blowing-up solution, 
Friedmann-Lema\^itre-Robertson-Walker spacetime

%%%%%%%%%%%%%%%%%%%%%%%%%%%%%%%
%
%%%%%%%%%%%%%%%%%%%%%%%%%%%%%%%

\section{Introduction}

We consider the Cauchy problem of semi-linear Klein-Gordon equations in Friedmann-Lema\^itre-Robertson-Walker spacetimes (FLRW spacetimes for short).
FLRW spacetimes are solutions of the Einstein equations 
with the cosmological constant under the cosmological principle.
They describe the spatial expansion or contraction, 
and yield some important models of the universe.
Let $n\ge1$ be the spatial dimension, 
$a(\cdot)$ be a scale-function defined on an interval $[0,T_0)$ for some 
$0< T_0\le\infty$, $c>0$ be the speed of light.
The metrics $\{g_{\alpha\beta}\}$ of FLRW spacetimes are expressed by 
\begin{equation}
 \label{Intro-RW}
 -c^2(d\tau)^2
 =
 \sum_{0\le \alpha,\beta\le n}g_{\alpha\beta}dx^\alpha dx^\beta
 :=
 -c^2(dt)^2+
 a(t)^{2}
 \sum_{j=1}^n (dx^j)^2,
\end{equation}
where we have put the spatial curvature as zero,  
the variable $\tau$ denotes the proper time, 
$x^0=t$ is the time-variable 
(see e.g., 
\cite{Carroll-2004-Addison, DInverno-1992-Oxford}).
When $a(\cdot)=1$ or $a(t)=e^{Ht}$ for the Hubble constant $H\in \br$, 
the spacetime with \eqref{Intro-RW} reduces to the Minkowski spacetime or the de Sitter spacetime, respectively.

%%%%%%%%%%%%%%%%%%%%%%%
%
% Derivation of equations
%
%%%%%%%%%%%%%%%%%%%%%%

We recall the derivation of the Klein-Gordon equation.
We denote the first and second derivatives of one variable function $a$ by $\dot a$
and $\ddot a$.
The Klein-Gordon equation generated by the above metric 
$(g_{\alpha\beta})_{0\le \alpha,\beta\le n}$ is given by 
$-(\sqrt{|g|})^{-1} \partial_\alpha (\sqrt{|g|} g^{\alpha\beta}\partial_\beta u)+m^2u=f(u)$
for the determinant $g:=\det (g_{\alpha\beta})$ and the inverse matrix $(g^{\alpha\beta})$,
i.e.,  
\beq
 \label{Cauchy-0}
 \partial_t^2u+\frac{n\dot a}{a}\partial_tu-\frac{c^2}{a^2} \Delta u 
 +m^2c^2u=c^2 f(u),
\eeq
where 
$\Delta:=\sum_{j=1}^n \partial_j^2$ denotes the Laplacian. 

In this paper,
we consider the Cauchy problem of \eqref{Cauchy-0} given by
\beq
 \label{Cauchy}
 \left\{
 \begin{array}{l}
  \partial_t^2u+n\dot a(t)a(t)^{-1}\partial_tu-c^2a(t)^{-2} \Delta u +m^2c^2u
  =c^2 f(u),
  \\
  u(t_0,\cdot)=u_0(\cdot),\ \ \partial_tu(t_0,\cdot)=u_1(\cdot)
 \end{array}
 \right.
\eeq
for $(t,x)\in [t_0,T)\times\br^n$
with
$0\le t_0<T\le T_0$, where
$u_0$, $u_1$ are given initial data. For semi-linear terms, we assume that 
$f : \bc \rightarrow \bc$ satisfies that there exists $F : \bc \rightarrow \br$ and $\varepsilon>0$ such that 
 \beq
  \label{f-condition}  
 \partial_t F(u)=\Re \left\{f(u)\partial_t \overline{u}\right\} \ \ 
  \mbox{and} \ \ 
  \Re \left\{f(u)\overline{u}\right\} \ge (2+\vep)F(u)
 \eeq
for any $u : [t_0,T) \rightarrow \bc$.
We note that if we set
$f : \br \rightarrow \br$,
then the condition on \eqref{f-condition} is described as
\[
   F(u):=\int_{0}^s f(\xi)d\xi \ \ 
   \mbox{and} \ \ 
  f(u)u\ge(2+\varepsilon)F(u) 
\]
for $u : [t_0,T)\rightarrow\br$, which has been considered in \cite{McCollum-Mwamba-Oliver-2024-NA, Yang-Xu-2018-ApplMathLetters}. The assumption \eqref{f-condition} includes the power-type non-linearity. If we take $f(u)=|u|^{p-1}u$ for $u\in\bc$ with $p>1$, then we have to set $\varepsilon\le p-1$. If we take $f(u)=-|u|^{p-1}u$ for $u\in\bc$ with $p>1$, then we have to set $\vep\ge p-1$. For another example, we note that $f(u)=\pm|u|^{p}$ for $u\in\br$ with $p>1$, then we have to set $\vep=p-1$.

We denote the Lebesgue space by $L^q(I)$ for an interval $I\subset \br$ 
and $1\le q\le \infty$ with the norm 
\[
\|Y\|_{L^q(I)}:=
\begin{cases}
\left\{
\int_I |Y(t)|^q dt
\right\}^{1/q} & \mbox{if}\ \ 1\le q<\infty,
\\
{\mbox{ess.}\sup}_{t\in I} |Y(t)| 
& \mbox{if}\ \ q=\infty.
\end{cases}
\]
We denote by $\|\cdot\|$ the $L^2(\brn)$ norm, and by $\|\cdot\|_q$ the $L^q(\brn)$ norm for $q\neq2$.
We also denote $(u,v):=\int_\brn u\overline{v} dx$ by the $L^2$-inner product.

We define the energy functional 
 \begin{equation}
  \label{Def-Energy}
  E(t):=
  \frac12\|u_t\| ^2+\frac12c^2a^{-2}\|\nabla u\| ^2
  +\frac12m^2c^2\|u\| ^2-c^2\int_\brn F(u) dx,
 \end{equation}
the Nehari functional
 \beq
  \label{Def-I}
  I(u):=c^2a^{-2}\|\nabla u\|^2+m^2c^2\|u\|^2-c^2\Re\int_\brn \overline{u}f(u) dx
 \eeq
and the unstable set
 \beq
  \label{Def-B}
  \mathcal{B}:=\{ u \in C([t_0,T), H^1(\brn))\cap C^1([t_0,T),L^2(\brn)) \, ; \, I(u)<0\}
 \eeq
for the Cauchy problem of \eqref{Cauchy}.
We note that the Nehari functional $I(u)$ depends on time.
The set $\mathcal{B}$ is invariant under our assumptions for initial data.
More precisely, if $I(u_0)<0$ holds, then we have $u\in\mathcal{B}$ (see, Lemma \ref{B-invariant} and Lemma \ref{B-invariant-IP}, below).

%%%%%%%%%%%%%%%%%%%%%%%%%%%%%%%%%%%%%%%%%
%
%    table
%
%%%%%%%%%%%%%%%%%%%%%%%%%%%%%%%%%%%%%%%%%%%%%

\vspace{-11pt}

\begin{table}[h]
 \caption{Initial data leading to high energy blowup of the problem \eqref{Cauchy}.}
 \label{table}
 \centering
 \small
  \begin{tabular}{clll}
   \hline   
   \multicolumn{3}{c}{
   Under $I(u_0)<0$} 
                          & Progress \\
   \hline    \vspace{-5.5pt} \\
   Case {\rm{I}} \ \  
                           & $\frac{\widetilde{m}^2\widetilde{c}^2\vep}{2(\vep+2)}\|u_0\|^2
                                 >E(t_0)$
                               \ \  \ \ 
                           & $E(t_0)\ge
                                \frac{\widetilde{m}^2\widetilde{c}^2\vep}{2(\vep+2)}
                                                                                \Re(u_0,u_1)\ge0$                       
                           & Solved in this paper 
                              \vspace{5.5pt}\\
   Case {\rmtwo} \ \  
                             & $\frac{\widetilde{m}^2\widetilde{c}^2\vep}{2(\vep+2)}\|u_0\|^2
                                   >E(t_0)$
                             & $\frac{\widetilde{m}^2\widetilde{c}^2\vep}{2(\vep+2)}                
                                  \Re(u_0,u_1)>E(t_0)\ge0$ 
                             & Solved in this paper 
                                \vspace{5.5pt}\\
   Case {\rmthree} \ \ 
                               & $E(t_0)\ge
                                    \frac{\widetilde{m}^2\widetilde{c}^2\vep}{2(\vep+2)}\|u_0\|^2$
                               & $\frac{\widetilde{m}^2\widetilde{c}^2\vep}{2(\vep+2)}
                                    \Re(u_0,u_1)>E(t_0)\ge0$  
                               & Solved in this paper 
                                \vspace{5.5pt}\\
   Case {\rmfour} \ \  
                             & $E(t_0)\ge
                                  \frac{\widetilde{m}^2\widetilde{c}^2\vep}{2(\vep+2)}\|u_0\|^2$  
                             & $E(t_0)\ge
                                  \frac{\widetilde{m}^2\widetilde{c}^2\vep}{2(\vep+2)}
                                                                                  \Re(u_0,u_1)\ge0$  
                             & Still open
                               \vspace{5.5pt}\\
   \hline
  \end{tabular}
\end{table}
%\vspace{-11pt}

In \cite{Nakamura-2014-JMAA}, the existence of global solutions was proved for semi-linear Klein-Gordon equations in de Sitter spacetime with $n\le4$. This result was extended to the case of general FLRW spacetimes in \cite{Galstian-Yagdjian-2015-NA-TMA, Nakamura-Yoshizumi-9999}.
Galstian and Yagdjian in \cite{Galstian-Yagdjian-2015-NA-TMA} also proved the local well-posedness with $n\ge1$ for semi-linear Klein-Gordon equations in FLRW spacetimes for any $u_0\in H^1(\brn)$, $u_1\in L^2(\brn)$ with $\dot a(\cdot)>0$. 
%So that, we obtain the existence of the local solution under some of assumptions of Theorem \ref{Thm-blow-1} and Theorem \ref{Thm-blow-B}. 

The nonexistence of the global solution of the Klein-Gordon equation in the Minkowski spacetime (i.e., $\dot a\equiv0$) was shown by Levine in \cite{Levine-1974-TransAMS} when the initial energy was negative. We also refer to the result \cite{ZhangJian-2002-NA}, which established a sharp condition for blowup of one in the Minkowski spacetime when the initial energy had an upper bound. When the initial energy was arbitrarily high, the first result was a condition which Wang in \cite{Y. Wang-2008-PAMS} showed for blowing-up solutions corresponding to the Klein-Gordon equation in the Minkowski spacetime.
This result was extended by Yang and Xu in \cite{Yang-Xu-2018-ApplMathLetters}. %Both of these were considered in the Minkowski spacetime (i.e., $\dot a\equiv0$). 

On the other hand, for FLRW spacetimes,
the case of the gauge invariant semi-linear term of the form $f(u)=\lambda|u|^{p-1}u$
with $1<p<\infty$, $\lambda\in\bc$ was considered in \cite{Nakamura-Yoshizumi-9999} for $\dot a\le0$ by the concavity method for $a^2\|u\|^2$.   
Moreover, McCollum, Mwamba and Oliver in \cite{McCollum-Mwamba-Oliver-2024-NA} showed a condition for  blowing-up solutions of \eqref{Cauchy} for large initial data with the positive initial energy.

Our work is to solve Case ${\rm{I}}$, Case ${\rmtwo}$ and Case ${\rmthree}$ in Table \ref{table} under $\widetilde{m}>0$ and $\widetilde{c}>0$ defined by \eqref{m-c-tilde}, below.
Case ${\rmfour}$ is still open.

We show the following theorem and corollary for blowing-up solutions of \eqref{Cauchy} for large initial data.
Put $a_0:=a(0)$ and $a_1:=\dot a(0)$.

%%%%%%%%%%%%%%%%%%%%%%%%%%%%%%%
%
%   Theorem-blowing-up
%
%%%%%%%%%%%%%%%%%%%%%%%%%%%%%%%%%

\begin{theorem}\label{Thm-blow-1}
Let $m\in\br$.
Let $f\in C(\bc,\bc)$ be a function with \eqref{f-condition} for some $\vep>0$.
Assume that $t_0=0$, and $u_0\in H^1(\brn)$, $u_1\in L^2(\brn)$ satisfy
 \beq
   \label{Thm-A-1000}
  \rho:=
  \frac{m^2c^2\varepsilon }{2(\varepsilon+2)}\|u_0\| ^2-E(0)
  >0
 \eeq
for $E(\cdot)$ given by \eqref{Def-Energy}.
Assume $\Re (u_0,u_1)\ge0$.
Let $0<T_0\le\infty$.
Let $a\in C^2([0,T_0),(0,\infty))$ satisfy
$\dot a(t)\ge0$ and
 \beq
   \label{Thm-A-4000}
  \dot a(t)^{2}-\ddot a(t)a(t)\ge0
 \eeq
for any $t\in[0,T_0)$.
Put
 \beq
  \label{Estimate-T_1}
    T:=\max\left\{1,
            \frac{\pi^2(1+na_1a_0^{-1})\|u_0\|^2}              
                   {\vep^2\rho}  
           \right\}.
 \eeq
Then, the solution $u$ of \eqref{Cauchy} blows-up in $L^2(\brn)$ in finite time no later than $T$ if $T\le T_0$.
More precisely, $\|u\|\rightarrow\infty$ as $t\nearrow T_*$ for some $T_*$
with $0<T_*\le T$.
\end{theorem}

%%%%%%%%%%%%%%%%%%%%%%%%%%%%%%%
%
%   Remark
%
%%%%%%%%%%%%%%%%%%%%%%%%%%%%%%%%

\begin{remark}

Let $f\in C(\bc,\bc)$ satisfy the following conditions

\beq
\label{Rem-1000}
f(0)=0, \ \ \ 
|f(s)-f(v)|\le C|s-v|\left(|s|^{p-1}+|v|^{p-1}\right)
\eeq
for any $s,v\in \bc$ and some constant $C>0$ independent of $s, v$, and 
\beq
\label{Rem-2000}
1<p<
\begin{cases}
\infty & \mbox{if \, } n=1,2, \\
1+\frac{2}{n-2} & \mbox{if \, } n\ge3.
\end{cases}
\eeq
Then, we note that the condition \eqref{Thm-A-1000} does not hold for sufficiently small data since $\rho$ is rewritten by 
 \begin{align*}
  \rho&=-\frac12\|u_1\|^2-\frac12c^2a^{-2}_0\|\nabla u_0\|^2
          -\frac{m^2c^2}{\vep+2}\|u_0\|^2+c^2\int_\brn F(u_0)dx \\
        &\le
            -\frac12\|u_1\|^2-\frac12c^2a^{-2}_0\|\nabla u_0\|^2
          -\frac{m^2c^2}{\vep+2}\|u_0\|^2+\frac{c^2}{\vep+2}\int_\brn |u_0|^{p+1}dx 
 \end{align*}
and we have $\rho\le0$ if $u_0$ is sufficiently small.

\end{remark}

%%%%%%%%%%%%%%%%%%%%%%%%%%%%%%%
%
%   Theorem-blowing-up
%
%%%%%%%%%%%%%%%%%%%%%%%%%%%%%%%%%

\begin{theorem}\label{Thm-blow-B}
Let $m\in\br$.
Let $f\in C(\bc,\bc)$ be a function with \eqref{f-condition} for some $\vep>0$.
Let $0\le t_0<T_0\le\infty$.
Let $a\in C^2([0,T_0),(0,\infty))$ satisfy
$\dot a(t)\ge0$ and \eqref{Thm-A-4000}
for any $t\in[t_0,T_0)$.
Assume %that there exists $t_0\ge0$ such that
 \beq
   \label{t_0-condition}
  \frac{\dot a(t_0)}{a(t_0)}
   \begin{cases}
    \le \frac{|m|c\left\{\sqrt{\vep(\vep+4)}-\vep\right\}}{2n} & \mbox{if} \ \ m\neq0, \\
    < \infty & \mbox{if} \ \ m=0.
   \end{cases}
 \eeq
%holds, and the initial time $t_0$ satisfies \eqref{t_0-condition}.
Let
$u_0\in H^1(\brn)$, $u_1\in L^2(\brn)$ satisfy
 \beq
   \label{Thm-B-1000}
  \delta:=
  \frac{|m|c\varepsilon }{2(\varepsilon+2)}\Re(u_0,u_1)-E(t_0)
  >0,
 \eeq
$I(u_0)<0$ and $\Re (u_0,u_1)\ge0$.
Put
 \beq
  \label{Estimate-T_delta}
    T:=t_0+\max\left[1,
           \frac{2\pi^2(\vep+4)\{1+n\dot a(t_0)a(t_0)^{-1}\}\|u_0\|^2}              
                  {\vep^2(\vep+2)\delta}  
           \right].
 \eeq
Then, the solution $u$ of \eqref{Cauchy} blows-up in $L^2(\brn)$ in finite time no later than $T$ if $T\le T_0$.
More precisely, $\|u\|\rightarrow\infty$ as $t\nearrow T_*$ for some $T_*$
with $t_0<T_*\le T$.
\end{theorem}

When $f\in C(\bc,\bc)$ satisfies \eqref{Rem-1000} with \eqref{Rem-2000},
we note that the condition \eqref{Thm-B-1000} does not hold for sufficiently small data
(see, Appendix in \cite{McCollum-Mwamba-Oliver-2024-NA}).

%%%%%%%%%%%%%%%%%%%%%%%%%%%%%%%
%
%   Remark
%
%%%%%%%%%%%%%%%%%%%%%%%%%%%%%%%%

\begin{remark}

%(1)
For $m\in\br$ and $c>0$, we put 
 \beq
   \label{m-c-tilde}
  \widetilde{m}:=\min\{1,|m|\}, \ \ \widetilde{c}:=\min\{1,c\}.
 \eeq
If we have 
 \begin{equation*}
  \frac{\widetilde{m}^2\widetilde{c}^2\varepsilon}{2(\varepsilon+2)}\|u_0\| ^2>E(t_0),
 \end{equation*}
then \eqref{Thm-A-1000} holds on $t_0=0$.
If we have
 \begin{equation*}
  \frac{\widetilde{m}^2\widetilde{c}^2\varepsilon }{2(\varepsilon+2)}\Re(u_0,u_1)>E(t_0),
 \end{equation*}
then \eqref{Thm-B-1000} holds.
So that, in this paper, we solve Case ${\rm{I}}$, Case ${\rmtwo}$ and Case ${\rmthree}$ in Table \ref{table}, below.

%(2)
%By \eqref{Thm-A-1000} which is equivalent to
% \begin{equation*}
%  \|u_1\| ^2+c^2a_0^{-2}\|\nabla u_0\| ^2
%  +\frac{2m^2c^2}{\varepsilon+2}\|u_0\| ^2
%  -\frac{2\lambda c^2}{\varepsilon+2}\|u_0\|_{\varepsilon+2}^{\varepsilon+2}
%  <0,
% \end{equation*}
%we cannot set small initial data. 
%Thus, Theorem \ref{Thm-blow-1} is the result for large initial data.
\end{remark}

In \cite{McCollum-Mwamba-Oliver-2024-NA},
a sufficient condition for blowing-up solutions of \eqref{Cauchy} was shown under $I(u_0)<-n\dot aa^{-1}\widetilde{m}\widetilde{c}\Re(u_0,u_1)<0$ and $\frac{\widetilde{m}^2\widetilde{c}^2\vep}{3(\vep+2)}\Re(u_0,u_1)>E(t_0)\ge0$ with $\widetilde{m},\, \widetilde{c}>0$. Since we have shown blowing-up solutions of \eqref{Cauchy} under $I(u_0)<0$ and $\frac{\widetilde{m}^2\widetilde{c}^2\vep}{2(\vep+2)}\Re(u_0,u_1)>E(t_0)\ge0$ with $\widetilde{m},\, \widetilde{c}>0$ in Theorem \ref{Thm-blow-B}, we improve the result in \cite{McCollum-Mwamba-Oliver-2024-NA} in terms of the upper bounds of $I(u_0)$ and $E(t_0)$. Moreover, we also have shown the blowing-up solution of \eqref{Cauchy} under $I(u_0)<0$ and $\frac{\widetilde{m}^2\widetilde{c}^2\vep}{2(\vep+2)}\|u_0\|>E(t_0)\ge0$. Blowup of \eqref{Cauchy} in this condition was not shown in \cite{McCollum-Mwamba-Oliver-2024-NA}.

%This table is also used in \cite{Yang-Xu-2018-ApplMathLetters} and \cite{McCollum-Mwamba-Oliver-2024-NA} to illustrate their results. 
%Case ${\rm{I}}$, and both Case ${\rmtwo}$ and Case ${\rmthree}$ under $-n\dot aa^{-1}\widetilde{m}\Re(u_0,u_1)\le I(u_0)<0$ have been open. 

We adapt the concavity method which was first introduced by Levine in \cite{Levine-1974-TransAMS, Levine-MA-1974}.
In the Minkowski spacetime, this method was employed for the function $m^2\|u\|^2$ in \cite{Levine-1974-TransAMS, Levine-MA-1974, Y. Wang-2008-PAMS, Yang-Xu-2018-ApplMathLetters}. McCollum, Mwamba and Oliver in \cite{McCollum-Mwamba-Oliver-2024-NA} also employed this method for $m^2\|u\|^2$ in FLRW spacetimes. Instead of the function $m^2\|u\|^2$, we employ the concavity method for a new function $\theta$ given by 
 \[
  \theta(t):=
   \|u\| ^2+\int_{t_0}^t n\left\{\dot a(\tau)a(\tau)^{-1}\|u\| ^2+G(\tau)\right\}d\tau
   +n(T-t)\dot a(t_0)a(t_0)^{-1}\|u_0\| ^2
 \]
for $t\in[t_0,T)$ with some $T\in(t_0,T_0]$, where
 \[
   G(t):=
    \int_{t_0}^t \left\{\frac{\dot a(\tau)^{2}-\ddot a(\tau)a(\tau)}{a(\tau)^{2}}\right\}\|u\| ^2
    d\tau.
 \] 
By the second and third terms of $\theta$,  
we improve the condition on $I(u_0)$, and we solve Case ${\rm{I}}$.
Moreover, we can consider the condition on $m=0$ and $E(t_0)<0$ in Theorem \ref{Thm-blow-1} and Theorem \ref{Thm-blow-B}.
%So that, we note that our results contain the case of $m=0$, but we have $E(t_0)<0$ when $m=0$.
This yields that our results also include blowing-up solutions for semi-linear wave equations in FLRW spacetimes with the negative energy.
Setting the third term of $\theta$ is inspired by \cite{Gazzola-Squassina} and \cite{Xu-Ding-2013-AMS}.
Gazzola and Squassina in \cite{Gazzola-Squassina}, and Xu and Ding in \cite{Xu-Ding-2013-AMS}
showed some conditions for blowing-up solutions of damped wave, and Klein-Gordon equations in the Minkowski spacetime.

Now, we introduce some concrete examples of the scale-function $a$.
For $\sigma\in\mathbb{R}$ and the Hubble constant $H\in \br$, 
we put
\begin{equation}
\label{R-Def-T_0}
T_0:= 
\begin{cases}
\infty & \mbox{if}\ \ (1+\sigma)H\ge0,
\\
-\frac{2}{n(1+\sigma)H} & \mbox{if}\ \ (1+\sigma)H<0,
\end{cases}
\end{equation}
and define $a(\cdot)$ by 
\begin{equation}
\label{Def-a}
a(t):=
\begin{cases}
a_0\left\{1+\frac{n(1+\sigma)Ht}{2}\right\}^{2/n(1+\sigma)} & \mbox{if}\ \ \sigma\ne-1,
\\
a_0\exp(Ht) & \mbox{if}\ \ \sigma=-1
\end{cases}
\end{equation}
for $0\le t<T_0$. 
We note $a_0=a(0)$ and $H=\dot{a}(0)/a(0)$, where $\dot a:=da/dt$.
This scale-function $a(\cdot)$ describes the Minkowski spacetime when $H=0$ 
(namely, $a(\cdot)$ is a constant $a_0$),
the expanding space when $H>0$ with $\sigma\ge -1$,
the blowing-up space when $H>0$ with $\sigma< -1$ 
(the ``Big-Rip'' in cosmology), 
the contracting space when $H<0$ with $\sigma\le -1$, 
and the vanishing space when $H<0$ with $\sigma> -1$ 
(the ``Big-Crunch'' in cosmology).
It describes the de Sitter spacetime when $\sigma=-1$ 
(see, e.g., \cite{Nakamura-2020-OsakaJMath}).

We obtain the following corollaries from the above theorems, respectively for the concrete example of $a$ given by \eqref{Def-a}.

%%%%%%%%%%%%%%%%%%%%%%%%%%%%%%%%%%%%%%%%%%%%%%%%%%
%
%  Cor-blowing-up
%
%%%%%%%%%%%%%%%%%%%%%%%%%%%%%%%%%%%%%%%%%%%%%%%%%

\begin{corollary}
 \label{Cor-blowing-up}
Let $f\in C(\bc,\bc)$ be a function with \eqref{f-condition} for some $\vep>0$.
Assume that $t_0=0$, $m\in\br$, and $u_0\in H^1(\brn)$, $u_1\in L^2(\brn)$ satisfy
\eqref{Thm-A-1000} and $\Re(u_0,u_1)\ge0$.
Let $T_0$ and $a$ be defined by \eqref{R-Def-T_0} and \eqref{Def-a}.
Assume one of the following conditions (i) and (ii) holds.

(i)
$H=0$, $\sigma\in\br$.

(ii)
$H>0$, $\sigma\ge-1$.
\\
Then, the condition on \eqref{Thm-A-4000} is satisfied, and the result in Theorem \ref{Thm-blow-1} hold.
Namely, the solution $u$ of the Cauchy problem \eqref{Cauchy} blows-up in finite time no later than $T$ defined by \eqref{Estimate-T_1}.
\end{corollary}

%%%%%%%%%%%%%%%%%%%%%%%%%%%%%%%%%%%%%%%%%%%%%%%%%%%%%%%%%%%%%%%
%
%
%
%%%%%%%%%%%%%%%%%%%%%%%%%%%%%%%%%%%%%%%%%%%%%%%%%%%%%%%%%%%%%%%%

\begin{corollary}
 \label{Cor-blowing-up-B}
Let $f\in C(\bc,\bc)$ be a function with \eqref{f-condition} for some $\vep>0$.
Assume that $m\in\br$, and $u_0\in H^1(\brn)$, $u_1\in L^2(\brn)$ satisfy
\eqref{Thm-B-1000}, $I(u_0)<0$ and $\Re(u_0,u_1)\ge0$.
We put 
 \begin{equation}
  \label{Def-C}
  C_\vep:=\frac{2}{|m|c\{\sqrt{\vep(\vep+4)}-\vep\}},
 \end{equation}
for $m\neq0$.
Let $T_0$ and $a$ be defined by \eqref{R-Def-T_0} and \eqref{Def-a}.
Assume one of the following conditions (i), (ii), (iii) and (iv) holds.

(i)
$m\in\br$,
$H=0$, $\sigma\in\br$, $t_0=0$.

(ii)
$m=0$, $H>0$, $\sigma\ge-1$, $t_0=0$.

(iii)
$m\neq0$,
$0<H\le 1/nC_\vep$, $\sigma\ge-1$, $t_0=0$.

(iv)
$m\neq0$,
$H>1/nC_\vep$, $\sigma>-1$, $t_0=2C_\vep/(1+\sigma)-2/n(1+\sigma)H(>0)$
\\
Then, the conditions \eqref{Thm-A-4000} and \eqref{t_0-condition} are satisfied, and the result in Theorem \ref{Thm-blow-B} hold.
Namely, the solution $u$ of the Cauchy problem \eqref{Cauchy} blows-up in finite time no later than $T$ defined by \eqref{Estimate-T_delta}.

\end{corollary}

In Corollaries \ref{Cor-blowing-up} and \ref{Cor-blowing-up-B}, 
the case $H=0$ reduces to the semi-linear Klein-Gordon equation in the Minkowski spacetime, 
which was extensively studied 
(see, e.g., 
\cite{Ball-1978-AcademicPress, Levine-1974-TransAMS, Y. Wang-2008-PAMS, Xu-Ding-2013-AMS, Yang-Xu-2018-ApplMathLetters, ZhangJian-2002-NA} 
on blowing-up solutions, and the references therein). 
We included this case to compare it with the case $H\neq0$.

This paper is organized as follows.
In Section \ref{Sec-Pre}, we collect fundamental properties on some ordinary differential equation of second order, the variant space $I(u)<0$, the energy $E$ and the concavity function.
In Section \ref{Sec-Thm-blow-1}, \ref{Sec-Thm-blow-B}, \ref{Sec-Cor-blowing-up} and \ref{Sec-Cor-blowing-up-B}, we prove Theorem \ref{Thm-blow-1}, Theorem \ref{Thm-blow-B}, Corollary \ref{Cor-blowing-up} and Corollary \ref{Cor-blowing-up-B}.

%%%%%%%%%%%%%%%%%%%%%%%%%%%%%%%%%%%%%%%%%%%%%%%
%
%  Preliminaries
%
%%%%%%%%%%%%%%%%%%%%%%%%%%%%%%%%%%%%%%%%%%%%%%%%

\newsection{Preliminaries}
 \label{Sec-Pre}
We prepare several lemmas to prove the results in the previous section.
We start with the following fundamental statement for $E(t)$ and $I(u)$ defined by \eqref{Def-Energy} and \eqref{Def-I}, respectively.

\begin{lemma}
 \label{Lem-E-I}
Let $T>0$, $a_0>0$, $\dot a\ge0$, $u_0\in H^1(\brn)$ and $u_1\in L^2(\brn)$.
If $u$ is the solutions to the Cauchy problem \eqref{Cauchy}, then the following results hold for any $t\in[t_0,T)$.

(1)
 \beq
  \label{Rel-E-I}
  E(t)\ge\frac12 \|u_t\|^2+\frac{I(u)}{\varepsilon+2}+\frac{\varepsilon}{2(\varepsilon+2)}
         \left(c^2a^{-2}\|\nabla u\|^2+m^2c^2\|u\|^2\right).
 \eeq

(2)
 \beq
  \label{eq-Energy}
  E(t)+
  \int_{t_0}^t
   \left(n\dot aa^{-1}\|u_t\| ^2+c^2\dot aa^{-3}\|\nabla u\| ^2\right)
  d\tau
  =E(t_0)
 \eeq
and $E(t)\le E(t_0)$.

(3)
\beq
   \label{eq-I}
 \Re (u,u_{tt})+n\dot aa^{-1} \Re (u,u_t)+I(u)=0.
 \eeq

\end{lemma}

\begin{proof}

(1)
By \eqref{f-condition}, \eqref{Def-Energy} and \eqref{Def-I}, we have 
\begin{align*}
\frac{c^2a^{-2}\|\nabla u\|^2+m^2c^2\|u\|^2-I(u)}{\vep+2}
&=
\frac{c^2}{\vep+2}\Re \int_\brn \overline{u}f(u)dx  \\
&\ge
c^2\int_\brn F(u)dx \\
&=
\frac12\|u_t\| ^2+\frac12c^2a^{-2}\|\nabla u\| ^2
  +\frac12m^2c^2\|u\| ^2-E(t).
\end{align*}
This yields the inequality of \eqref{Rel-E-I}.

(2)
Multiplying $\overline{u_t}$ to the first equation of \eqref{Cauchy},
integrating it for $x\in\brn$, taking its real part 
and integrating it for $[t_0,t]$,
we obtain the equation of \eqref{eq-Energy}.
By $a_0>0$, $\dot a\ge0$ and \eqref{eq-Energy}, we have $E(t)\le E(t_0)$.

(3)
Multiplying $\overline{u}$ to the first equation of \eqref{Cauchy}, integrating it for $x\in\brn$
and taking its real part, we obtain the equation of \eqref{eq-I}.
\end{proof}

%%%%%%%%%%%%%%%%%%%%%%%%%%%%%%%%%%%%%%%%%%%%%%%
%
%
%
%%%%%%%%%%%%%%%%%%%%%%%%%%%%%%%%%%%%%%%%%%%%%%%%%

Next, we set a lemma with the fundamental statement for the ordinary differential
inequality. 

\begin{lemma}
 \label{Lem-ODI}
Let $T>0$, $\gamma\in C^1([t_0,T),\br)$, 
and $h\in C^1([t_0,T),\br)$. Assume that $h(t_0)\ge0$ and 
 \beq
  \label{Lem-ODI-1000}
  h'(t)+\gamma'(t)h(t)>0
 \eeq
for any $t\in[t_0,T)$. Then, $h(t)>0$ for any $t\in(t_0,T)$.
\end{lemma}

\begin{proof}
Multiplying $e^{\gamma(t)}>0$ to the both sides of \eqref{Lem-ODI-1000},
and noting $(h'+\gamma'h)e^{\gamma}=(he^{\gamma})'$,
we have
 \begin{equation*}
  \left\{h(t)e^{\gamma(t)}\right\}'>0
 \end{equation*}
for any $t\in[t_0,T)$.
Integrating its both sides for $[t_0,t)$, we obtain
 \begin{align*}
  h(t)e^{\gamma(t)}
  &>h(t_0)e^{\gamma(t_0)} \\
  &\ge0
 \end{align*}
for any $t\in[t_0,T)$ by $h(t_0)\ge0$.
This yields $h(t)>0$ for any $t\in(t_0,T)$.
\end{proof}

Here, we put 
 \beq
  \label{Def-L}
  L(t):=\|u(t)\|^2
 \eeq
for $t\in[t_0,T)$.
We prove that the set $\mathcal{B}$ is invariant under our assumptions.

%%%%%%%%%%%%%%%%%%%%%%%%%%%%%%%%%%%%%%%%
%
%
%
%%%%%%%%%%%%%%%%%%%%%%%%%%%%%%%%%%%%%%%%%%%%%%

\begin{lemma}
  \label{B-invariant}
Let $0\le t_0<T$, $a(t_0)>0$ and $\dot a\ge0$.
Let $u_0\in H^1(\brn)$, $u_1\in L^2(\brn)$ satisfy $\Re(u_0,u_1)\ge0$ and 
\eqref{Thm-A-1000}.
Then, $I(u_0)<0$, and 
the solution $u$ of \eqref{Cauchy} belongs to $\mathcal{B}$.
\end{lemma}

\begin{proof}
 Firstly, we prove $I(u_0)<0$.
 By \eqref{Thm-A-1000} and \eqref{Rel-E-I}, we have
 \begin{equation*}
  \frac12\|u_1\|^2+\frac{I(u_0)}{\varepsilon+2}
 +\frac{\varepsilon}{2(\varepsilon+2)}c^2a(t_0)^{-2}\|\nabla u_0\|^2
  <0.
 \end{equation*}
 This yields $I(u_0)<0$.

 Next, we prove $I(u(t))<0$ for any $t>t_0$.
 When $m=0$, we obtain $E(t_0)<0$ by \eqref{Thm-A-1000}.
 Since we have $E(t)\le E(t_0)$ by \eqref{eq-Energy} and $\dot a\ge0$, we obtain $E(t)<0$.
 Since we have $E(t)\ge I(u)/(\varepsilon+2)$ by \eqref{Rel-E-I}, we obtain $I(u)<0$.
 Next, we consider the case of $m\neq0$.
 If  $I(u(t))<0$ for any $t\in[t_0,T)$ does not hold, then there exists a first $t_*\in(t_0,T)$ 
 such  that $I(u(t_*))=0$ by $I(u_0)<0$ and the continuity of $I(u)$ on $[t_0,T)$.
 Here, we have
  \beq
   \label{Proof-Lem-Bi-1000}
   L'=2\Re(u,u_t), \ \ L'(t_0)=2\Re(u_0,u_1)\ge0
  \eeq
 and
  \beq
   \label{Proof-Lem-Bi-2000}
   L''
   =2\|u_t\| ^2+2\Re(u,u_{tt}),
  \eeq
 where $L$ has been defined by \eqref{Def-L}.
 By \eqref{eq-I}, \eqref{Proof-Lem-Bi-1000} and \eqref{Proof-Lem-Bi-2000},
 we obtain
  \begin{align*}
   L''
    &=2\left\{\|u_t\| ^2-I(u)-n\dot aa^{-1}\Re(u,u_t)\right\} \\
    &=2\left\{\|u_t\| ^2-I(u)\right\}-n\dot aa^{-1} L'.
  \end{align*}
 This yields
  \begin{equation}
    \label{Proof-Lem-Bi-3000}
   L''+n\dot aa^{-1} L'=2\left\{\|u_t\| ^2-I(u)\right\}.
  \end{equation}
 Noting that $I(u)<0$ for $t\in[t_0,t_*)$, we have 
  \begin{equation*}
   L''+n\dot aa^{-1}L'>0
  \end{equation*}
 for $t\in[t_0,t_*)$.
 Replacing $T$, $\gamma(t)$, $h(t)$ in Lemma \ref{Lem-ODI} with $t_*$, $n\log{a(t)}$, $L'(t)$,
 we obtain $L'(t)>0$ for any $t\in(t_0,t_*)$ by $L'(t_0)\ge0$ in \eqref{Proof-Lem-Bi-1000}. Thus, this yields $L(t_*)>L(t_0)$
 since $L$ is continuous on $[t_0,t_*]$. By $L(t_*)>L(t_0)$ and \eqref{Thm-A-1000}, we have 
  \begin{equation}
   \label{Proof-Lem-Bi-4000}
   L(t_*)>\frac{2(\varepsilon+2)}{m^2c^2\varepsilon }E(t_0).
  \end{equation} 
 In contrast, we have 
  \begin{align*}
   E(t_*)
    &\ge \frac12\|u_t(t_*)\|^2+\frac{\varepsilon}{2(\varepsilon+2)}
        \left(c^2a(t_*)^{-2}\|\nabla u(t_*)\|^2+m^2c^2\|u(t_*)\|^2\right) \\
    &\ge \frac{m^2c^2\varepsilon }{2(\varepsilon+2)}L(t_*)
  \end{align*}
 by \eqref{Rel-E-I} and $I(u(t_*))=0$.
 Since we have $E(t_0)\ge E(t_*)$ by (2) of Lemma \ref{Lem-E-I}, we have 
  \begin{equation*}
    E(t_0)\ge \frac{m^2c^2\varepsilon }{2(\varepsilon+2)}L(t_*).
  \end{equation*}
 This leads to a contradiction to \eqref{Proof-Lem-Bi-4000}.
 Thus, we have proved $I(u(t))<0$ for any $t\in[t_0,T)$.
 This yields $u\in\mathcal{B}$.
\end{proof}

%%%%%%%%%%%%%%%%%%%%%%%%%%%%%%%%%%%%%%%%%%%%%
%
%
%
%%%%%%%%%%%%%%%%%%%%%%%%%%%%%%%%%%%%%%%%%%%%%%%%

\begin{lemma}
 \label{B-invariant-IP}
Let $a>0$, $\dot a\ge0$ and \eqref{Thm-A-4000}. Assume that there exists $t_0\in[0,T)$ such that \eqref{t_0-condition} holds.
Assume that $u_0\in H^1(\brn)$, $u_1\in L^2(\brn)$ satisfy $I(u_0)<0$ and 
\eqref{Thm-B-1000}.
Then,
the solution $u$ of \eqref{Cauchy} belongs to $\mathcal{B}$.
 
\end{lemma}

\begin{proof}
 When $m=0$, the proof is similar to Lemma \ref{B-invariant}.
 We consider the case of $m\neq0$ in the following. Putting
  \begin{equation*}
   H(t):=L'(t)-\frac{4(\vep+2)}{|m|c\vep}E(t_0)=2\Re(u,u_t)-\frac{4(\vep+2)}{|m|c\vep}E(t_0),
  \end{equation*}
 we have
  \begin{align}
   \label{Proof-Lem-BiI-1}
   H'&=L'' \nonumber \\
      &= 2\left\{\|u_t\| ^2-I(u)\right\}-n\dot aa^{-1} L' \nonumber\\
      &\ge \left(\vep+4\right)\|u_t\|^2+\vep m^2c^2\|u\|^2-2(\vep+2)E(t)
              -n\dot aa^{-1}L'  \nonumber\\
      &\ge 2|m|c\sqrt{\vep(\vep+4)}\|u\|\|u_t\|-2(\vep+2)E(t_0)-n\dot aa^{-1}L' 
         \nonumber\\
%      &> |m|c\sqrt{\vep(\vep+4)}\|u\|\|u_t\|-2(\vep+2)E(t_0)-n\dot aa^{-1}L' \nonumber\\
      &\ge |m|c\sqrt{\vep(\vep+4)}|\Re(u,u_t)|-2(\vep+2)E(t_0)-n\dot aa^{-1}L' 
         \nonumber\\
      &= \left\{|m|c\sqrt{\vep(\vep+4)}-2n\dot aa^{-1}\right\}|\Re(u,u_t)|-2(\vep+2)E(t_0)
         \nonumber\\
      &\ge |m|c\vep\Re(u,u_t)-2(\vep+2)E(t_0) 
         \nonumber\\
      &=\frac{|m|c\vep}2H
  \end{align}
 for any $t\in[t_0,T)$,
 by \eqref{Proof-Lem-Bi-3000}, \eqref{Rel-E-I}, $x^2+y^2\ge2xy$ for any $x, y\in\br$
 and $E(t)\le E(t_0)$ for any $t\in[t_0,T)$, where we have used
 $|m|c\sqrt{\vep(\vep+4)}-2n\dot aa^{-1}\ge |m|c\vep$ by \eqref{t_0-condition} and 
 $\frac{d}{dt}\dot aa^{-1}\le0$ derived from \eqref{Thm-A-4000}.
 Since we have $H(t_0)>0$ by \eqref{Thm-B-1000}, we obtain $H(t)>0$ 
 for any $t\in[t_0,T)$ by putting $h:=H$ and $\gamma:=-|m|c\vep t/2$, in
 Lemma \ref{Lem-ODI}. This yields $L''>0$ for any $t\in[t_0,T)$. Thus, we have
  \begin{align}
    \label{Proof-Lem-BiI-1.5}
   \Re(u,u_t)=\frac{L'(t)}2
                &>\frac{L'(t_0)}2=\Re(u_0,u_1)  \\ \label{Proof-Lem-BiI-2}
                &>\frac{2(\vep+2)}{|m|c\vep}E(t_0) 
  \end{align}
 for any $t\in[t_0,T)$ by $L''>0$ and \eqref{Thm-B-1000}.
 Here, if  $I(u(t))<0$ for any $t\in[0,T)$ does not hold, 
 then there exists a first $t_*\in(t_0,T)$ such  that $I(u(t_*))=0$ by $I(u_0)<0$ 
 and the continuity of $I(u)$ on $[t_0,T)$.
 Then, we have
  \begin{align*}
    %\label{Proof-Lem-BiI-3}
   E(t_0)&\ge E(t_*) \nonumber \\
          &\ge \frac12\|u_t(t_*)\|^2
                 +\frac{\vep}{2(\vep+2)}(c^2a(t_*)^{-2}\|\nabla u(t_*)\|+m^2c^2\|u(t_*)\|^2)
               \nonumber \\
           &\ge \frac{\vep}{2(\vep+2)}\|u_t(t_*)\|^2+\frac{m^2c^2\vep}{2(\vep+2)}\|u(t_*)\|^2
               \nonumber \\
           &\ge \frac{|m|c\vep}{2(\vep+2)}\|u(t_*)\|\|u_t(t_*)\| \nonumber \\
           &\ge \frac{|m|c\vep}{2(\vep+2)}\Re(u(t_*),u_t(t_*))
  \end{align*}
 by $E(t_0)\ge E(t)$ for any $t\in[t_0,T)$, \eqref{Rel-E-I} and $I(u(t_*))=0$.
 This leads a contradiction to \eqref{Proof-Lem-BiI-2}. Thus, we have proved $I(u(t))<0$
 for any $t\in[t_0,T)$. This yields $u\in\mathcal{B}$. 
\end{proof}

%%%%%%%%%%%%%%%%%%%%%%%%%%%%%%%%%%%%%%%%%%%%
%
%
%
%%%%%%%%%%%%%%%%%%%%%%%%%%%%%%%%%%%%%%%%%%%%%%

\begin{lemma}
  \label{Lem-Thm-A}
Let $0\le t_0<T$, $a(t_0)>0$ and $\dot a\ge0$. 
Let $u_0\in H^1(\brn)$, $u_1\in L^2(\brn)$ satisfy $\Re(u_0,u_1)\ge0$ 
and \eqref{Thm-A-1000}. Let $u$ be the solution of \eqref{Cauchy}. 
Then, 
 \beq
  \label{Lem-Thm-A-1000}
  L'(t)>0 \ \  \mbox{and} \ \ 
  m^2c^2\vep L(t)-2(\varepsilon+2)E(t_0)>2(\vep+2)\rho
 \eeq
holds for any $t\in(t_0,T)$, where $\rho$ is defined by \eqref{Thm-A-1000}.
\end{lemma}

\begin{proof}
 By Lemma \ref{B-invariant}, we have $I(u(t))<0$ for any $t\in[t_0,T)$. 
 By \eqref{Proof-Lem-Bi-3000} and $I(u)<0$, we have $L''(t)+n\dot a(t)a(t)^{-1}L'(t)>0$ for  
 any $[t_0,T)$.
 Using Lemma \ref{Lem-ODI}, we obtain $L'(t)>0$ for any $t\in(t_0,T)$ by $L'(t_0)=2\Re (u_0,u_1)\ge0$.
 By $L'>0$ and \eqref{Thm-A-1000}, we have 
  \begin{align*}
   m^2c^2\vep L(t)-2(\vep+2)E(t_0)
    &>m^2c^2\vep L(t_0)-2(\vep+2)E(t_0) \\
    &=2(\vep+2)\rho
  \end{align*}
 for any $t\in(t_0,T)$.
\end{proof}

%%%%%%%%%%%%%%%%%%%%%%%%%%%%%%%%%%%
%
%  Lem-concavity
%
%%%%%%%%%%%%%%%%%%%%%%%%%%%%%%%%%%%%%%%%

In the following, we show the key lemma in order to use the concavity method.
\begin{lemma}
 \label{Lem-concavity}
Let $\kappa$, $A$ and $B$ be positive constants.
Let $0\le t_0<T$ satisfy
 \beq
  \label{Lem-T-condition}
  T > t_0+\frac{\pi^2(2\kappa+1)B}{8\kappa^2 A}.
 \eeq
We consider the following differential inequality ;
 \begin{equation}
  \label{Deq-concavity}
  \left\{
  \begin{array}{l}
    y''(t)\le -\kappa A y(t)^{1+1/\kappa},
   \\
    y(t_0)=y_0, \ \ y'(t_0)=y_1
 \end{array}
 \right.
 \end{equation}
for any $t\in[t_0,T)$.
If $y_0$ and $y_1$ satisfy
 \begin{equation}
  \label{Lem-initial}
  \left\{B(T-t_0)\right\}^{-\kappa}\le y_0, \ \ y_1\le0,
 \end{equation}
then there exists $T_*\in(t_0,T)$ such that 
the solution $y\in C^2([t_0,T),\br)$ of \eqref{Deq-concavity} satisfies
$y\searrow0$ as $t\nearrow T_*$.
\end{lemma}

\begin{proof}
 Assume $y(t)>0$ for any $t\in[t_0,T)$. We show a contradiction.
 We note that we have $y''<0$ by $y>0$ and \eqref{Deq-concavity}, and
 we also have $y'(t)<0$ for any $t\in(t_0,T)$ by $y''<0$ and $y_1\le0$.
 Multiplying $y'<0$ to the both sides of the first inequality in \eqref{Deq-concavity}, 
 we obtain
  \begin{equation*}
   \frac12\left\{(y')^2\right\}'
    \ge
   -\frac{\kappa^2 A}{2\kappa+1}\left\{y^{2+1/\kappa}\right\}'.
  \end{equation*}  
 This yields
  \begin{align}
   \label{  Proof-Lem-C-1000}
   (y')^2
    &\ge
   y_1^2+{\rm{I}}y_0^{2+1/\kappa}-{\rm{I}}y^{2+1/\kappa}  \nonumber\\
    &>
   y_1^2+{\rm{I}}y_0^{2+1/\kappa}-{\rm{I}}y_0^{1/\kappa}y^{2}
  \end{align}
 for $t\in(t_0,T)$ by $0<y<y_0$,
 where we have put 
  \begin{equation*}
   {\rm{I}}:=\frac{2\kappa^2 A}{2\kappa+1}.
  \end{equation*}
 Noting that the right hand side of \eqref{  Proof-Lem-C-1000} is positive by $y_0>y$, 
 we have 
  \begin{equation*}
   y'<-\left({\rmtwo}-{\rmthree}y^2\right)^{1/2}
  \end{equation*}
 by $y'<0$, where we have put
  \[
   {\rmtwo}
    := y_1^2+{\rm{I}}y_0^{2+1/\kappa},
      \ \  
   {\rmthree}
    :={\rm{I}}y_0^{1/\kappa}.
  \]
 This yields
  \begin{equation}
   \label{  Proof-Lem-C-2000}
   \frac{y'}{\left({\rmtwo}-{\rmthree}y^2\right)^{1/2}}<-1.
  \end{equation}
 Putting $z:=({\rmthree}/{\rmtwo})^{1/2}y$ and integrating the both sides of
 \eqref{  Proof-Lem-C-2000} on $[t_0,t]$ with $t_0<t<T$, we have $z>0$ and 
  \begin{equation}
   \label{  Proof-Lem-C-5000000}
   \int_{t_0}^t {\rmthree}^{-1/2}\frac{d}{d\tau}\Sin z(\tau)d\tau<-(t-t_0),
  \end{equation}
 where $\Sin z(\cdot)$ is the inverse to $z(\cdot)=\sin\varphi$ 
 for $\varphi\in[-\pi/2,\pi/2]$.
 This yields
  \begin{equation*}
   {\rmthree}^{1/2}(t-t_0)-\Sin z(t_0)<-\Sin z(t)<0
  \end{equation*}
 by $z>0$ and we have
  \beq
    \label{Proof-Lem-C-4000}
   \left|{\rmthree}^{1/2}(t-t_0)-\Sin z(t_0)\right|<\left|\Sin z(t)\right|<\frac \pi2
  \eeq
 by $\Sin z(\cdot)\in[-\pi/2,\pi/2]$.
 We have
  \[
   \Sin z(t)<-{\rmthree}^{1/2}(t-t_0)+\Sin z(t_0).
  \]
 by \eqref{  Proof-Lem-C-5000000}.
 This yields
  \begin{equation*}
   z(t)<-\sin\left({\rmthree}^{1/2}(t-t_0)-\Sin z(t_0)\right)
  \end{equation*}
 by noting \eqref{Proof-Lem-C-4000}.
 Putting
  \beq
    \label{Proof-Lem-C-T_*}
   \widetilde{T}_*:=t_0+\frac{\Sin z(t_0)}{{\rmthree}^{1/2}},
  \eeq
 where we note $\widetilde{T}_*>t_0$ by
  \begin{equation}
     \label{  Proof-Lem-C-5000}
   \Sin z(t_0)=\Sin \left\{\left(
                  \frac{{\rm{I}}y_0^{2+1/\kappa}}{y_1^2+{\rm{I}}y_0^{2+1/\kappa}}
                        \right)^{1/2}\right\}
    \ \ \in[0,\pi/2],
 \end{equation}
 there exists $T_*\in(t_0,\widetilde{T}_*]$ such that $z\searrow0$ as $t\nearrow T_*$, which yields 
 $y\searrow0$ as $t\nearrow T_*$.
 Since we have 
  \begin{align*}
   \label{Proof-Lem-C-4}
   \widetilde{T}_*-t_0
     &=
    {\rm{I}^{-1/2}}y_0^{-1/2\kappa}
    \Sin \left\{\left(
                  \frac{{\rm{I}}y_0^{2+1/\kappa}}{y_1^2+{\rm{I}}y_0^{2+1/\kappa}}
                        \right)^{1/2}\right\}  
      \nonumber \\
     &\le
    \frac{\pi{\rm{I}^{-1/2}}y_0^{-1/2\kappa}}2
      \nonumber \\
     &=
    \left\{\frac{\pi^2(2\kappa+1)}{8\kappa^2 A}\right\}^{1/2}y_0^{-1/2\kappa}
      \nonumber \\
     &\le
    \left\{\frac{\pi^2(2\kappa+1)}{8\kappa^2 A}\right\}^{1/2}B^{1/2}(T-t_0)^{1/2}
       \\
     &<
    (T-t_0)^{1/2}(T-t_0)^{1/2}=T-t_0
      \nonumber
  \end{align*}
 by \eqref{Proof-Lem-C-T_*}, \eqref{  Proof-Lem-C-5000}, \eqref{Lem-initial}
 and \eqref{Lem-T-condition}, we have $T_*\le \widetilde{T}_*<T$.
\end{proof}

%%%%%%%%%%%%%%%%%%%%%%%%%%%%%%%%%%%%%%%%
%
%
%
%%%%%%%%%%%%%%%%%%%%%%%%%%%%%%%%%%%%%%%%%%%

\newsection{Proof of theorems and corollaries}

We prove Theorem \ref{Thm-blow-1}, Theorem \ref{Thm-blow-B}, Corollary \ref{Cor-blowing-up} and Corollary \ref{Cor-blowing-up-B} in this section.
\subsection{Proof of Theorem \ref{Thm-blow-1}}
 \label{Sec-Thm-blow-1}

Let $t_0=0$.
We remember that we have put $a_0:=a(0)$, $a_1:=\dot a(0)$
and
 \begin{equation*}
  \label{Def-T_1}
  T:=\max\left\{1,
            \frac{\pi^2(1+na_1a_0^{-1})\|u_0\|^2}              
                   {\varepsilon^2\rho}  
           \right\}.
 \end{equation*}
We consider the Cauchy problem of \eqref{Cauchy} for $(t,x)\in[0, T)\times\brn$.
For any $t\in[0,T)$, we put
 \beq
  \label{A-3}
  G(t):=
    \int_0^t \left\{\frac{\dot a(\tau)^{2}-\ddot a(\tau)a(\tau)}{a(\tau)^{2}}\right\}\|u(\tau)\| ^2
    d\tau
 \eeq
and
 \begin{equation}
  \label{A-4}
  \theta(t):=
   \|u(t)\| ^2+\int_0^t n\left\{\dot a(\tau)a(\tau)^{-1}\|u(\tau)\| ^2+G(\tau)\right\}d\tau
   +n(T-t)a_1a^{-1}_0\|u_0\| ^2. 
 \end{equation}
We note that $\theta(\cdot)>0$ since we have $G(\cdot)\ge0$ by \eqref{Thm-A-4000}.
By \eqref{A-4}, we have
 \begin{align}
   \label{A-6}
  \theta'
   &=2\Re(u,u_t)+n\dot aa^{-1}\|u\| ^2
       -na_1a^{-1}_0\|u_0\| ^2+nG  \nonumber\\
   &=2\Re(u,u_t)
      +\int_0^t\partial_\tau\left(n\dot aa^{-1}\|u\| ^2\right)d\tau
      +nG  \nonumber\\
   &=2\Re(u,u_t)
     +2\Re\int_0^t n\dot aa^{-1}(u,u_t)d\tau
 \end{align}
and
 \begin{align}
   \label{A-9}
  \theta''
  &= 2\Re(u,u_{tt})+2n\dot aa^{-1}\Re(u,u_t)+2\|u_t\| ^2 \nonumber \\
  &=2\left\{\|u_t\| ^2-I(u)\right\}
  \end{align}
by \eqref{eq-I}.
Since we have $I(u)<0$ by Lemma \ref{B-invariant}, we obtain $\theta''>0$, which yields
 \beq
   \label{A-7.1}
  \theta'(t)>\theta'(0)=2\Re(u_0,u_1)\ge0
 \eeq
for any $t\in(0,T)$ by \eqref{A-6} and the assumption $\Re(u_0,u_1)\ge0$.
By \eqref{A-6}, we have 
 \begin{multline}
   \label{A-10}
  \frac{\left(\theta'(t)\right)^2}4
  =-\eta(t)+\left\{\theta(t)-\int_0^t nG(\tau)d\tau-(T-t)na_1a_0^{-1}\|u_0\| ^2
                 \right\}\\
    \cdot
    \left\{\|u_t\| ^2+\int_0^t n\dot a(\tau)a(\tau)^{-1}\|u_t\| ^2d\tau\right\},
 \end{multline}
where we have put
 \begin{align}
   \label{A-11}
  \eta(t)
  &:=\left\{\|u\| ^2+\int_0^t n\dot a(\tau)a(\tau)^{-1}\|u\| ^2d\tau\right\}
              \left\{\|u_t\| ^2+\int_0^t n\dot a(\tau)a(\tau)^{-1}\|u_t\| ^2d\tau\right\}
    \nonumber \\
  &  \ \ \ \ \ \ \ \ \ \ \ \ \ \ \ \ \ \ \ \ \ \ \ \ \ \ \ \ \ \ \ \ \ \ \ \ \ \ 
      -\left\{\Re(u,u_t)+\Re\int_0^t n\dot a(\tau)a(\tau)^{-1}(u,u_t)d\tau\right\}^2
     \nonumber \\
  &=\|u\|^2\|u_t\|^2-\left\{\Re(u,u_t)\right\}^2  \nonumber \\
  &\ \ \ \ \ \ 
     +\left(\int_0^t n\dot aa^{-1}\|u\| ^2d\tau\right)
       \left(\int_0^t n\dot aa^{-1}\|u_t\| ^2d\tau\right)
     -\left\{\Re \int_0^t n\dot aa^{-1}(u,u_t)d\tau\right\}^2 \nonumber \\
  & \ \ \ \ \ \ 
     +\|u\| ^2\int_0^t n\dot aa^{-1}\|u_t\| ^2d\tau 
             +\|u_t\| ^2\int_0^t n\dot aa^{-1}\|u\| ^2d\tau \nonumber \\
  & \ \ \ \ \ \ \ \ \ \ \ \ \ \ \ 
     -2\left\{\Re(u,u_t)\right\}     
        \left\{\Re \int_0^t n\dot aa^{-1}(u,u_t)d\tau\right\}. 
 \end{align}
We note that we have
 \beq
   \label{A-12}
  \Re(u,u_t)
  \le
  \int_\brn |u\overline{u_t}|dx
  \le
  \|u\| \|u_t\| 
 \eeq
and
 \begin{align}
   \label{A-13}
  \left|\Re \int_0^t n\dot aa^{-1}(u,u_t)d\tau\right|^2
   &\le
    \left(\int_0^t n\dot aa^{-1}\|u\| \|u_t\|d\tau
     \right)^2\nonumber\\
   &\le
    \left(\int_0^t n\dot aa^{-1}\|u\| ^2d\tau\right)
    \left(\int_0^t n\dot aa^{-1}\|u_t\| ^2d\tau\right)
 \end{align}
by the H\"{o}lder inequality.
Using \eqref{A-12} and \eqref{A-13}, we have
 \begin{align}
   \label{A-14}
  \biggl\{\Re
    (u,u_t)\biggr\}
    &\left\{
      \Re \int_0^t n\dot aa^{-1}(u,u_t)d\tau
      \right\}\nonumber\\
    &\le
      \|u\| \|u_t\| 
      \left(\int_0^t n\dot aa^{-1}\|u\| ^2d\tau\right)^{1/2}
      \left(\int_0^t n\dot aa^{-1}\|u_t\| ^2d\tau\right)^{1/2}\nonumber\\
    &\le
      \frac12
      \left(\|u\| ^2\int_0^t n\dot aa^{-1}\|u_t\| ^2d\tau 
             +\|u_t\| ^2\int_0^t n\dot aa^{-1}\|u\| ^2d\tau
      \right),
 \end{align}
where we have used $xy\le(x^2+y^2)/2$ for any $x,y\in\br$ at the last inequality.
By \eqref{A-11}, \eqref{A-12}, \eqref{A-13} and \eqref{A-14}, we have
 \beq
   \label{A-15}
  \eta(t)\ge0
 \eeq
for any $t\in[0,T)$.
So that, we have
 \beq
   \label{A-16}
  \frac{\left\{\theta'(t)\right\}^2}4
  \le
  \theta(t)
  \left\{\|u_t\| ^2+\int_0^t n\dot a(\tau)a(\tau)^{-1}\|u_t\| ^2d\tau\right\}
 \eeq
by $\dot a(\cdot)\ge0$, $G(\cdot)\ge0$ and \eqref{A-15}.
By \eqref{A-16} and \eqref{A-9},
we obtain
 \begin{align}
   \label{A-17}
  \theta(t)\theta''(t)-\frac{\widetilde{\kappa}+3}4\left\{\theta'(t)\right\}^2
   &\ge
  \theta(t)\left\{
               \theta''(t)-(\widetilde{\kappa}+3)  \left(
                  \|u_t\| ^2+\int_0^t n\dot aa^{-1}\|u_t\| ^2d\tau\right)\right\}
     \nonumber \\
   &=
  \theta(t)\zeta(t)
 \end{align}
for any $\widetilde{\kappa}>0$,
where we have put
 \begin{equation}
   \label{A-18}
  \zeta(t):=
  -(\widetilde{\kappa}+1)\|u_t\| ^2-2I(u) 
  -(\widetilde{\kappa}+3)\int_0^t n\dot a(\tau)a(\tau)^{-1}\|u_t\| ^2d\tau.
 \end{equation}
Here, let 
\[
\widetilde{\kappa}=\vep+1.
\]
By \eqref{Rel-E-I}, $\dot a\ge0$, \eqref{eq-Energy} and Lemma \ref{Lem-Thm-A}, we have
 \begin{align}
   \label{A-19}
  \zeta(t)
  &\ge  \vep(c^2a^{-2}\|\nabla u\|^2+m^2c^2\|u\|^2)-2(\vep+2)E(t) 
           -(\vep+4)\int_0^t n\dot aa^{-1}\|u_t\| ^2d\tau   \nonumber \\
  &=-2(\varepsilon+2)\left\{E(t)+\int_0^t n\dot aa^{-1}\|u_t\| ^2d\tau
                \right\}  \nonumber \\
  &\ \ \ \ \ \ \ \ \ \ \ \ \ 
    \ \ \ \ \ \ \ \ \ \ \ \ \ 
    +\varepsilon \int_0^t n\dot aa^{-1}\|u_t\| ^2d\tau 
    +\varepsilon \left(c^2a^{-2}\|\nabla u\| ^2+m^2c^2\|u\| ^2\right)\nonumber\\
  &\ge -2(\varepsilon+2)\left\{E(t)+\int_0^t n\dot aa^{-1}\|u_t\| ^2d\tau\right\}
      +m^2c^2\vep\|u\|^2  \nonumber \\
  &= -2(\varepsilon+2)\left\{E(0)-\int_0^t c^2\dot aa^{-3}\|\nabla u\|^2d\tau\right\}
      +m^2c^2\varepsilon \|u\| ^2 \nonumber\\
  &\ge -2(\varepsilon+2)E(0)+m^2c^2\vep\|u\|^2  \nonumber \\
  &= -2(\varepsilon+2)E(0)+m^2c^2\varepsilon L(t)  \nonumber \\
  &\ge-2(\varepsilon+2)E(0)+m^2c^2\varepsilon L(0) \nonumber \\
  &=2(\varepsilon+2)\rho 
 \end{align}
for any $t\in[0,T)$, 
where $\rho$ and $L$ are defined by \eqref{Thm-A-1000} and \eqref{Def-L}, and we have used $L(t)\ge L(0)$ by the first inequality $L'(t)>0$ for any $t\in(0,T)$ in \eqref{Lem-Thm-A-1000}.
So that, we have
 \beq
   \label{A-21}
  \theta(t)\theta''(t)-\frac{\vep+4}4\left\{\theta'(t)\right\}^2
  \ge2(\varepsilon+2)\rho\theta(t)
 \eeq
for any $t\in[0,T)$ by \eqref{A-17}, \eqref{A-19} and $\theta>0$.
Putting $\kappa:=\varepsilon /4$, we obtain
 \begin{align}
   \label{A-22}
  \left\{\theta^{-\kappa}\right\}''(t)
  &=
    -\kappa\theta(t)^{-\kappa-2}
    \left[\theta(t)\theta''(t)-\frac{\vep+4}4\left\{\theta'(t)\right\}^2\right] \nonumber\\
  &\le-2\kappa(\varepsilon+2)\rho\theta(t)^{-\kappa-1} \nonumber \\
  &=-2\kappa(\varepsilon+2)\rho\left\{\theta(t)^{-\kappa}\right\}^{1+1/\kappa}
 \end{align}
by $\theta>0$ and \eqref{A-21}.
Putting 
\[
y(t):=\theta(t)^{-\kappa}
\]
 and $A:=2(\varepsilon+2)\rho$, we have
\beq
\label{A-22.0}
y''(t)\le -\kappa Ay(t)^{1+1/\kappa}
\eeq
for any $t\in[0,T)$ by \eqref{A-22}.
Since we have
 \[
    0<\theta(0)=\|u_0\|^2+na_1a_0^{-1}T\|u_0\|^2
              \le (1+na_1a_0^{-1})T\|u_0\|^2
 \]
by $T\ge1$, we obtain
\beq
 \label{A-22.1}
y_0\ge (BT)^{-\kappa},
\eeq
where we have put $y_0:=y(0)$ and $B:=(1+na_1a_0^{-1})\|u_0\|^2$.
Since we have 
 \[
  \left\{\theta^{-\kappa}\right\}'(0)
  =-\kappa\theta'(0)\theta^{-\kappa-1}(0)\le0
 \]
by $\theta'(0)\ge0$ and $\theta(0)>0$, we obtain
\beq
\label{A-22.2} 
y_1:=y'(0)\le0.
\eeq
Moreover, we note that we have
\beq
 \label{A-22.3}
\frac{\pi^2(2\kappa+1)B}{8\kappa^2 A}
=
\frac{\pi^2(1+na_1a_0^{-1})\|u_0\|^2}{2\vep^2\rho}
<T
\eeq
by $\kappa=\vep/4$ and \eqref{Estimate-T_1}.
Since $T$ is defined by \eqref{Estimate-T_1}, and $y=\theta^{-\kappa}$ satisfy \eqref{Lem-T-condition}, \eqref{Deq-concavity} and \eqref{Lem-initial} by \eqref{A-22.3}, \eqref{A-22.0}, \eqref{A-22.1} and \eqref{A-22.2},
there exists $T_*\in(0,T)$ such that
$y(t)=\theta(t)^{-\kappa}\searrow0$ as $t\nearrow T_*$ by Lemma \ref{Lem-concavity}. From this fact, we obtain the following result as required in this theorem.

\begin{claim}\label{Claim-blowup}
 $\|u(t)\| \nearrow \infty$ as $t\nearrow T_*$.
\end{claim}

\begin{proof}
Since we have  
 \begin{align*}
  G(t)&=\int_0^t \left\{\frac{\dot a(\tau)^{2}-\ddot a(\tau)a(\tau)}{a(\tau)^{2}}\right\}
               \|u(\tau)\| ^2d\tau          \\
  %     &=\int_0^t \frac{d}{d\tau}\left(-\frac{\dot a}{a}\right)\|u\| ^2d\tau \\
       &\le \|u(t)\| ^2\int_0^t \frac{d}{d\tau}\left(-\frac{\dot a}{a}\right)d\tau \\
       &\le \{a_1a_0^{-1}-\dot a(t)a(t)^{-1}\}\|u(t)\| ^2,
 \end{align*}
where we have used \eqref{Thm-A-4000}
and $\|u\|^2$ is non-decreasing by Lemma \ref{Lem-Thm-A},
we obtain
 \begin{align*}
  \theta(t)
  &=
   \|u(t)\| ^2+
   \int_0^t n\left\{\dot a(\tau)a(\tau)^{-1}\|u(\tau)\| ^2+G(\tau)\right\}d\tau
   +n( T-t)a_1a^{-1}_0\|u_0\| ^2 \\
  &\le
   \|u(t)\| ^2
   +\int_0^t na_1a_0^{-1}\|u(\tau)\| ^2d\tau
   +n( T-t)a_1a^{-1}_0\|u_0\| ^2 \\
  &\le 
    \|u(t)\| ^2+ na_1a_0^{-1}t\|u(t)\| ^2
    +n( T-t)a_1a^{-1}_0\|u_0\| ^2.
 \end{align*}
Since $\theta\nearrow\infty$ as $t\nearrow T_*$, we have the claim.
\end{proof}

%%%%%%%%%%%%%%%%%%%%%%%%%%%%%%%%%%%%%%%%%%%%%%%%%%%%%%%
%
%
%
%%%%%%%%%%%%%%%%%%%%%%%%%%%%%%%%%%%%%%%%%%%%%%%%%%%%%%

\subsection{Proof of Theorem \ref{Thm-blow-B}}
 \label{Sec-Thm-blow-B}

We remember the definition of $T$ by \eqref{Estimate-T_delta}.
We define $G$, $\theta$ by
 \begin{equation*}
  G(t):=
    \int_{t_0}^t 
     \left\{\frac{\dot a(\tau)^{2}-\ddot a(\tau)a(\tau)}{a(\tau)^{2}}\right\}\|u(\tau)\| ^2
    d\tau
 \end{equation*}
and
 \begin{multline}
  \label{Def-Theta-2}
  \theta(t):=
   \|u(t)\| ^2+\int_{t_0}^t n\left\{\dot a(\tau)a(\tau)^{-1}\|u(\tau)\| ^2+G(\tau)\right\}d\tau\\
   +n(T-t)\dot a(t_0)a(t_0)^{-1}\|u_0\| ^2 
 \end{multline}
for any $t\in[t_0,T)$ by the similar definition of \eqref{A-3} and \eqref{A-4}.
Then, the similar argument holds
from \eqref{A-6} to \eqref{A-16}, where we have replaced $[0,T)$ with $[t_0,T)$. 
Similarly to the derivation of \eqref{A-17}, we obtain 
%\beq
%   \label{B-00}
%  \frac{\left\{\theta'(t)\right\}^2}4
%  \le
%  \theta(t)
%  \left\{\|u_t\| ^2+\int_{t_0}^t n\dot a(\tau)a(\tau)^{-1}\|u_t\| ^2d\tau\right\},
% \eeq
\begin{align}
   \label{B-1}
  \theta(t)\theta''(t)-\frac{\widetilde{\kappa}+3}4\left\{\theta'(t)\right\}^2
   &\ge
  \theta(t)\left[
               \theta''(t)-(\widetilde{\kappa}+3)  \left\{
                  \|u_t\| ^2+\int_{t_0}^t n\dot a(\tau)a(\tau)^{-1}\|u_t\| ^2d\tau\right\}\right]
     \nonumber \\
   &=
  \theta(t)\zeta(t)
 \end{align}
for any $\widetilde{\kappa}>0$,
where we have put
 \[
  \zeta(t):=
  -(\widetilde{\kappa}+1)\|u_t\| ^2-2I(u)
  -(\widetilde{\kappa}+3)\int_{t_0}^t n\dot a(\tau)a(\tau)^{-1}\|u_t\| ^2d\tau.
 \]
Here, let
 \[
  \widetilde{\kappa}=\frac{\vep}2+1.
 \]
By \eqref{Rel-E-I}, \eqref{eq-Energy}, $2\vep+1-\widetilde{\kappa}=3\vep/2>0$ and 
$\vep+1-\widetilde{\kappa}=\vep/2$, we have
 \begin{align}
   \label{B-2}
  \zeta(t)
  &\ge -(\widetilde{\kappa}+1)\|u_t\| ^2+(\vep+2)\|u_t\|^2
          +\vep(c^2a^{-2}\|\nabla u\|^2+m^2c^2\|u\|^2)         \nonumber \\
  & \ \ \ \ \ \ \ \ \ \ \ \ \ \ \ \ \ \ \ \ \ \ \ \ \ \ \ \ \ 
          -2(\vep+2)E(t)
          -(\widetilde{\kappa}+3)\int_{t_0}^t n\dot aa^{-1}\|u_t\| ^2d\tau
      \nonumber \\
  &=(\vep+1-\widetilde{\kappa})\|u_t\|^2+\vep(c^2a^{-2}\|\nabla u\|^2+m^2c^2\|u\|^2) 
      \nonumber \\
  & \ \ \ \ \ \ \ \ \  
          -2(\vep+2)\left\{E(t)+\int_{t_0}^t n\dot aa^{-1}\|u_t\| ^2d\tau\right\}
          +(2\vep+1-\widetilde{\kappa})\int_{t_0}^t n\dot aa^{-1}\|u_t\| ^2d\tau
      \nonumber \\
  &=(\vep+1-\widetilde{\kappa})\|u_t\|^2+\vep(c^2a^{-2}\|\nabla u\|^2+m^2c^2\|u\|^2) 
      \nonumber \\
  & \ \ \ \ \ \ \ \ \  
          -2(\vep+2)\left\{E(t_0)-\int_{t_0}^t c^2\dot aa^{-3}\|\nabla u\| ^2d\tau\right\}
          +(2\vep+1-\widetilde{\kappa})\int_{t_0}^t n\dot aa^{-1}\|u_t\| ^2d\tau
      \nonumber \\
  &\ge (\vep+1-\widetilde{\kappa})\|u_t\|^2+m^2c^2\vep\|u\|^2-2(\vep+2)E(t_0)
      \nonumber \\
  &= \frac{\vep}2\|u_t\|^2+m^2c^2\vep\|u\|^2-2(\vep+2)E(t_0) 
 \end{align}
for any $t\in[t_0,T)$. 
By \eqref{B-2}, $x^2+y^2\ge2xy$ for $x,y\in\br$, \eqref{Proof-Lem-BiI-1.5} and \eqref{Thm-B-1000},
we have
\begin{align}
 \label{B-2.5}
  \zeta(t)
  &\ge \frac{\vep}2\|u_t\|^2+m^2c^2\vep\|u\|^2-2(\vep+2)E(t_0) \nonumber\\
 &\ge \frac{\vep}2(\|u_t\|^2+m^2c^2\|u\|^2)-2(\vep+2)E(t_0)  \nonumber \\
  &\ge |m|c\vep\|u\| \|u_t\|-2(\vep+2)E(t_0) \nonumber \\
  &\ge |m|c\vep \Re(u,u_t)-2(\vep+2)E(t_0) \nonumber \\
  &\ge |m|c\vep \Re(u_0,u_1)-2(\vep+2)E(t_0) \nonumber \\
  &=2(\vep+2)\delta
\end{align}
for any $t\in[t_0,T)$.
By \eqref{B-1} and \eqref{B-2.5}, we have
 \beq
   \label{B-3}
  \theta(t)\theta''(t)-\frac{\widetilde{\kappa}+3}4\left\{\theta'(t)\right\}^2
   \ge
  2(\vep+2)\delta\theta(t)
 \eeq
for any $t\in[t_0,T)$. Putting $\kappa:=(\widetilde{\kappa}-1)/4$, we obtain
 \begin{align}
   \label{B-4}
  \left\{\theta^{-\kappa}\right\}''(t)
  &=
    -\kappa\theta^{-\kappa-2}(t)
    \left[\theta(t)\theta''(t)-\frac{\widetilde{\kappa}+3}4\left\{\theta'(t)\right\}^2\right] \nonumber\\
  &\le-2\kappa(\varepsilon+2)\delta\theta^{-\kappa-1}(t)
 \end{align}
by $\theta>0$ and \eqref{B-3}.
Putting
\[
y(t):=\theta(t)^{-\kappa}
\]
and $A:=2(\vep+2)\delta$, we have
\beq
\label{B-6}
y''(t)\le-2\kappa(\varepsilon+2)\delta\theta^{-\kappa-1}(t)
      =-\kappa Ay(t)^{1+1/\kappa}
\eeq
for any $t\in[t_0,T)$ by \eqref{B-4}.
We note that we have
 \[
  0<y(t_0)^{-1/\kappa}=\theta(t_0)%=\|u_0\|^2+n\dot a(t_0)a(t_0)^{-1}(T-t_0)\|u_0\|^2
              \le \left\{1+n\dot a(t_0)a(t_0)^{-1}\right\}(T-t_0)\|u_0\|^2=B(T-t_0)
 \]
by \eqref{Def-Theta-2} and $T-t_0\ge1$, and
 \[
  y'(t_0)=\left\{\theta^{-\kappa}\right\}'(t_0)
  =-\kappa\theta'(t_0)\theta(t_0)^{-\kappa-1}\le0,
 \]
by $\theta'(t_0)\ge0$ in \eqref{A-7.1}, where we have put $B:=\left\{1+n\dot a(t_0)a(t_0)^{-1}\right\}\|u_0\|^2$.
This yields
\beq
 \label{B-7}
 y_0\ge\left\{B(T-t_0)\right\}^{-\kappa}, \ \ 
 y_1\le0,
\eeq
where we have put $y_0:=y(t_0)$, $y_1:=y'(t_0)$.
Moreover, we note that we have
\beq
 \label{B-8}
\frac{\pi^2(2\kappa+1)B}{8\kappa^2 A}
=
\frac{\pi^2(\vep+4)\left\{1+n\dot a(t_0)a(t_0)^{-1}\right\}\|u_0\|^2}{\vep^2(\vep+2)\delta}
<T-t_0
\eeq
by \eqref{Estimate-T_delta}.
Since $T$ is defined by \eqref{Estimate-T_delta}, and $y=\theta^{-\kappa}$ satisfy 
\eqref{Lem-T-condition}, \eqref{Deq-concavity} and \eqref{Lem-initial} by \eqref{B-8}, \eqref{B-6} and \eqref{B-7},
there exists $T_*\in(t_0,T)$ such that
$y(t)=\theta(t)^{-\kappa}\searrow0$ as $t\nearrow T_*$ by Lemma \ref{Lem-concavity}. 
For $\|u\|\nearrow\infty$ as $t\nearrow T_*$, we use the similar argument for Claim \ref{Claim-blowup}.

%%%%%%%%%%%%%%%%%%%%%%%%%%%%%%%%%%%%%%%%%%%%%%%%%%%%%%%
%
%
%
%%%%%%%%%%%%%%%%%%%%%%%%%%%%%%%%%%%%%%%%%%%%%%%%%%%%%%%%%%

\subsection{Proof of Corollary \ref{Cor-blowing-up}}
 \label{Sec-Cor-blowing-up}

We check the assumptions $\dot a\ge0$ and \eqref{Thm-A-4000}.
By \eqref{Def-a}, we have 
 \begin{equation}
   \label{Proof-Cor-1}
  \frac{\dot a(t)}{a(t)}=H\left\{1+\frac{n(1+\sigma)Ht}2\right\}^{-1},
 \end{equation}
 \begin{equation*} 
  \frac{\ddot a(t)}{a(t)}=H^2\left\{1-\frac{n(1+\sigma)}2\right\}
                                     \left\{1+\frac{n(1+\sigma)Ht}2\right\}^{-2}.
 \end{equation*}
 This yields
 \begin{equation*}
  \frac{\dot a(t)^{2}-\ddot a(t)a(t)}{a(t)^{2}}
  =\frac{n(1+\sigma)H^2}2
         \left\{1+\frac{n(1+\sigma)Ht}2\right\}^{-2}.
 \end{equation*}
So that, we note that
the assumption \eqref{Thm-A-4000} holds if and only if $\sigma\ge-1$ under $H\neq0$, or $\sigma\in\br$ under $H=0$.  
Under $\sigma\ge-1$ and $H\neq0$, we note that $\dot a\ge0$ holds if and only if
$H>0$ by  \eqref{Proof-Cor-1}, $a>0$ and $1+n(1+\sigma)Ht/2>0$.
Under $\sigma\in\br$ and $H=0$, we note that $\dot a=0$ holds.
So that, the assumptions $\dot a\ge0$ and \eqref{Thm-A-4000} hold, and the required result follows from Theorem \ref{Thm-blow-1}.

%%%%%%%%%%%%%%%%%%%%%%%%%%%%%%%%%%%%%%%%%%%%%%%%%%%%%%%%%%%%%%%%
%
%
%
%%%%%%%%%%%%%%%%%%%%%%%%%%%%%%%%%%%%%%%%%%%%%%%%%%%%%%%%%%%%%%%%%

\subsection{Proof of Corollary \ref{Cor-blowing-up-B}}
 \label{Sec-Cor-blowing-up-B}

In the proof of Corollary \ref{Cor-blowing-up}, we have shown the assumptions $\dot a\ge0$ and \eqref{Thm-A-4000} hold if and only if $H=0$, $\sigma\in\br$, or $H>0$, $\sigma\ge-1$ hold. So that, we check the assumption \eqref{t_0-condition} in the following.
Firstly, let $t_0=0$. Then, we have
\[
\frac{\dot a(t_0)}{a(t_0)}=\frac{\dot a(0)}{a(0)}=H
\]
by \eqref{Proof-Cor-1}. So that, we have
\beq
 \label{Proof-Cor-2}
\eqref{t_0-condition} \iff 
H
\begin{cases}
\le \frac1{nC_\vep} & \mbox{if} \ \ m\neq0, \\
<\infty                  & \mbox{if} \ \ m=0,
\end{cases}
\eeq
where $C_\vep$ is defined by \eqref{Def-C}.
By $H=0$, $\sigma\in\br$, or $H>0$, $\sigma\ge-1$, and \eqref{Proof-Cor-2},
we have checked the cases of (i), (ii) and (iii) in Corollary \ref{Cor-blowing-up-B} satisfy all assumptions of Theorem \ref{Thm-blow-B}.

Next, let $t_0>0$. Since the cases of $m=0$ or $H\le 1/nC_\vep$ have been considered when $t_0=0$, we consider the case of $m\neq0$ and $H>1/nC_\vep$. Then, we have
\beq
 \label{Proof-Cor-3}
\eqref{t_0-condition} 
\iff 
  \frac{\dot a(t_0)}{a(t_0)}
  =H\left\{1+\frac{n(1+\sigma)Ht_0}2\right\}^{-1}
  \le \frac1{nC_\vep} 
\eeq
by \eqref{Proof-Cor-1}.
Noting that the last inequality in \eqref{Proof-Cor-3} does not hold when $\sigma=-1$ and $H>1/nC_\vep$,
we have
\[
\eqref{t_0-condition} 
\iff
 t_0\ge\frac{2C_\vep}{1+\sigma}-\frac2{n(1+\sigma)H}>0
  \ \ \mbox{and} \ \ \sigma>-1 
\]
under $H>1/nC_\vep$.
So that, we have checked the case of (iv) satisfies all assumptions of Theorem \ref{Thm-blow-B}.
Since all assumptions of Theorem \ref{Thm-blow-B} hold in the cases of (i), (ii), (iii) and (iv), the required result follows from Theorem \ref{Thm-blow-B}.

\vspace{10pt}

%%%%%%%%%%%%%%%%%%%%%%%%%%%%%%%%%%
%
%%%%%%%%%%%%%%%%%%%%%%%%%%%%%%%%%%

{\bf Acknowledgments.}
This work was supported by JSPS KAKENHI Grant Numbers 16H03940, 17KK0082, 22K18671, and by JST SPRING Grant Number JPMJSP2138.
%The authors are thankful to the anonymous referee for several comments to revise the paper.
%The author is thankful to the anonymous referee for several comments to revise the paper.

%{\bf Data Availability.}
%Data sharing is not applicable to this article as no new data were created or analyzed in this study.


\begin{thebibliography}{999}

\bibitem{Ball-1978-AcademicPress}
J. M. Ball, 
%Ball, J. M.
\emph{Finite time blow-up in nonlinear problems}, 
Nonlinear evolution equations, Proc. Symp., Madison/Wis. 1977, 189--205 (1978).

\bibitem{Carroll-2004-Addison}
S. Carroll, 
%Carroll, Sean(1-CHI)
\emph{Spacetime and geometry.  
An introduction to general relativity}, 
Addison Wesley, San Francisco, CA, 2004, xiv+513 pp.

\bibitem{DInverno-1992-Oxford}
R. d'Inverno, 
%d'Inverno, Ray(4-SHMP)
\emph{Introducing Einstein's relativity}, 
The Clarendon Press, Oxford University Press, New York, 1992, xii+383 pp. 

\bibitem{Galstian-Yagdjian-2015-NA-TMA} 
A. Galstian, K. Yagdjian, 
%Galstian, Anahit(1-UTRGV-M); Yagdjian, Karen(1-UTRGV-M)
\emph{Global solutions for semilinear Klein-Gordon equations in FLRW spacetimes},
%(English summary)
Nonlinear Anal. {\bf 113} (2015), 339--356.

\bibitem{Gazzola-Squassina}
F. Gazzola, M. Squassina,
\emph{Global solutions and finite time blow up for damped semilinear wave equations},
Ann. I. H. Poincar\'e, AN {\bf23} (2006), 185-207.

\bibitem{Levine-1974-TransAMS}
H. A. Levine,
\emph{Instability and nonexistence of global solutions of nonlinear wave equations of the form $Pu_{tt}=-Au+\mathcal{F}(u)$}, 
Trans. Amer. Math. Soc. {\bf 192} (1974), 1-21.

\bibitem{Levine-MA-1974}
H. A. Levine,
\emph{Some additional remarks on the nonexistence of global solutions to nonlinear wave equations},
SIAM J. Math. Anal. {\bf 5} (1974), 138-146.

\bibitem{McCollum-Mwamba-Oliver-2024-NA}
J. McCollum, G. Mwamba, J. Oliver,
\emph{A sufficient condition for blowup of the nonlinear Klein-Gordon equation with positive initial energy in FLRW spacetimes},
Nonlinear Anal. {\bf246} (2024), Paper No. 113582, 11 pp.

\bibitem{Nakamura-2014-JMAA}
M. Nakamura, 
%Nakamura, Makoto
\emph{The Cauchy problem for semi-linear Klein-Gordon equations in de Sitter spacetime},
J. Math. Anal. Appl.  {\bf 410}  (2014),  no. 1, 445--454.

\bibitem{Nakamura-2020-OsakaJMath} 
M. Nakamura,
\emph{Remarks on the derivation of several second order partial differential equations from a generalization of the Einstein equations},
Osaka J. Math., {\bf57} (2020), 305--331.

\bibitem{Nakamura-Yoshizumi-9999}
M. Nakamura, T. Yoshizumi,
\emph{The Cauchy problem for semi-linear Klein-Gordon  
equations in Friedmann-Lema\^itre-Robertson-Walker spacetimes},
J. Differential Equations, {\bf 438} 113395 (2025).

\bibitem{Y. Wang-2008-PAMS}
Y. Wang,
\emph{A sufficient condition for finite time blow up of the nonlinear Klein-Gordon equations with arbitrarily positive initial energy},
Proc. Amer. Math. Soc. {\bf136} (10) (2008), 3477-3482.

\bibitem{Xu-Ding-2013-AMS}
R. Xu, Y. Ding,
\emph{Global solutions and finite time blow up for damped Klein-Gordon equation},
Acta Math Scientia 2013, {\bf33B} (3), 643-652.

\bibitem{Yang-Xu-2018-ApplMathLetters}
Y. Yang, R. Xu,
\emph{Finite time blowup for nonlinear Klein-Gordon equations with arbitrarily positive initial energy},
Appl. Math. Lett. {\bf77} (2018), 21-26.

\bibitem{ZhangJian-2002-NA}
J. Zhang, 
%Zhang, Jian
\emph{Sharp conditions of global existence for nonlinear Schr\"odinger and Klein-Gordon equations},
Nonlinear Anal. {\bf 48} (2002), no. 2, 191--207.



\end{thebibliography}
\end{document}